\documentclass[11pt]{amsart}
\usepackage[ansinew]{inputenc}
\usepackage{amsmath,latexsym,amssymb,verbatim}
\usepackage{graphicx}
\usepackage{psfrag}
\usepackage{color}
\usepackage{enumerate}
\usepackage{enumitem}
\usepackage{makecell}
\newtheorem{lema}{Lemma}[section]
\newtheorem{theo}[lema]{Theorem}

\newcounter{teoremaganso}

\newtheorem {bigtheo} [teoremaganso] {Theorem}
\newtheorem{prop}[lema]{Proposition}

\theoremstyle{definition}
\newtheorem{defi}[lema]{Definition}
\theoremstyle{remark}
\newtheorem{rema}[lema]{Remark}

\newcommand{\disc}{\operatorname{disc}}
\newcommand{\res}{\operatorname{res}}

\newcommand{\atan}{\operatorname{atan}}
\newcommand{\Var}{\mathcal{V}}

\def\sideremark#1{\ifvmode\leavevmode\fi\vadjust{\vbox to0pt{\vss 
      \hbox to 0pt{\hskip\hsize\hskip1em           
 \vbox{\hsize2cm\tiny\raggedright\pretolerance10000
 \noindent #1\hfill}\hss}\vbox to8pt{\vfil}\vss}}}%

\title[Uniqueness of Limit Cycles for Quadratic Vector Fields
]{
Uniqueness of Limit Cycles for Quadratic Vector Fields
}
\author{J.L. Bravo, M. Fern\'{a}ndez, I. Ojeda, F. S\'{a}nchez}

\address{J.L. Bravo, Departamento de Matem\'{a}ticas,
Universidad de Extremadura, 06006 Badajoz, Spain}

\email{trinidad@unex.es}

\address{M. Fern\'{a}ndez, Departamento de Matem\'{a}ticas,
Universidad de Extremadura, 06006 Badajoz, Spain}

\email{ghierro@unex.es}

\address{I. Ojeda, Departamento de Matem\'{a}ticas,
Universidad de Extremadura, 06006 Badajoz, Spain}

\email{ojedamc@unex.es}

\address{F.S\'{a}nchez, Departamento de Matem\'{a}ticas,
	Universidad de Extremadura, 06006 Badajoz, Spain}

\email{fsanchez@unex.es}



\begin{document}

\begin{abstract}
This article deals with the study of the number of limit cycles surrounding a critical 
point of a quadratic planar vector field, which, in normal form, can be
written as $x'= a_1 x-y-a_3x^2+(2 a_2+a_5)xy + a_6 y^2$, 
$y'= x+a_1 y + a_2x^2+(2 a_3+a_4)xy -a_2y^2$. In particular, we study the semi-varieties defined
in terms of the parameters $a_1,a_2,\ldots,a_6$ where some classical criteria for 
the associated Abel equation apply. The proofs will combine classical ideas
with tools from computational algebraic geometry.
\end{abstract}

\maketitle

\section{Introduction and main results}

\noindent
The number of periodic solutions of a quadratic polynomial planar system is an open problem and the first non-trivial case of the second part of Hilbert's {\it XVI}-th problem. 

It is known that if a quadratic system has a limit cycle, i.e., a  periodic solution that is isolated in the set of periodic solutions of the system, then it must surround a focus of the system. In particular, if one takes the focus to be
at the origin, then the system can be written in the form (see \cite{Bautin})

\begin{equation}\label{system:quadratic}
	\begin{split}
x'&= a_1 x-y-a_3x^2+(2 a_2+a_5)xy + a_6 y^2, \\	
y'&= x+a_1 y + a_2x^2+(2 a_3+a_4)xy -a_2y^2.
\end{split}
\end{equation}	
One way to study the periodic solutions of \eqref{system:quadratic} is to analyse the $2\pi$-periodic positive solutions of the polar equation
 \begin{equation} \label{eq:polar}
 \frac{dr}{d\theta}=\frac{a_1r+f(\theta)r^2}{1+g(\theta)r},
\end{equation} 
 where
 $f$ and $g$ are the cubic homogeneous trigonometric polynomials defined by
 \begin{align*}
f(\theta)&= -a_3 \cos^3\theta+(3a_2+a_5) \cos^2\theta\sin\theta\\
& \phantom{=}+(2a_3+a_4+a_6)\cos\theta\sin^2\theta-a_2 \sin^3\theta,\\
g(\theta) &= a_2 \cos^3\theta+(3a_3+a_4)\cos^2\theta\sin\theta\\
&\phantom{=}-(3a_2+a_5)\cos\theta \sin^2\theta - a_6\sin^3\theta,	
 \end{align*}	
or of the Cherkas-equivalent Abel differential equation (see \cite{Cher})
 \begin{equation}\label{eq:Abel}
 \rho'=A(\theta)\rho^3 + B(\theta)\rho^2 + a_1 \rho,	
 \end{equation}	
where
\[
A(\theta)=g(\theta)(a_1 g(\theta)-f(\theta)),\quad B(\theta)=f(\theta)-2 a_1 g(\theta)-g'(\theta).
\]	
There are several results that establish upper bounds for the number of limit cycles of~\eqref{eq:Abel}.
The best known ones impose the condition that one of the functions $A$ or $B$ has definite sign, see  \cite{GG,GL,Ll,Panov, Pliss}, where
a $2\pi$-periodic function $F(\theta)$ has definite sign if $F(\theta)\geq 0$ for all $\theta \in[0,2\pi]$ or 
$F(\theta)\leq 0$ for all $\theta \in[0,2\pi]$.
  
In the particular case of Equation \eqref{eq:Abel}, the criteria in
\cite{GL, Pliss} give the following result.
\begin{theo}[\cite{GL,Pliss}]\label{teo:definite-sign}
If $A$ or $B$ has definite sign, 
then  Abel equation \eqref{eq:Abel} has at most one positive limit cycle.
\end{theo}

In \cite{CGLl} the quadratic systems for which the above criteria applies are described taking into account their number of critical points and the directions $\theta$ in which $g(\theta)=0$.

To establish our main results, which determine the semi-varieties in the space of parameters where the above criteria apply, and, as a consequence, to obtain that at most one limit cycle surrounds the origin of \eqref{system:quadratic}, we shall need the following notation. 

\medskip

The study of whether $A$ has definite sign can,
by the  change of variable $x=\tan(\theta)$ (see Section~\ref{sec:main}), be reduced to the study of the common roots of the polynomials $p_1(x)$ and $p_3(x):=a_1p_1(x)-p_2(x)$, where
\[\begin{split}
p_1(x)&=a_2+(3 a_3+a_4) x-(3 a_2+a_5) x^2-a_6 x^3,\\
p_2(x)&=- a_3 + (3 a_2 + a_5) x + (2 a_3 + a_4 + a_6) x^2 - a_2  x^3.
\end{split}\]

\medskip

Let us denote by $D_1,D_3,D_1',D_3'$ the discriminants of the polynomials $p_1,p_3,p_1',p'_3$, respectively.  
If $\res(p_1,p_3)$ denotes the resultant of $p_1$ and $p_3$ with respect to $x$, then it factorizes as 
\[
\res(p_1,p_3)=R_1 R_2,
\]
where
\[
R_1=(4a_2+a_5)^2+(3a_3+a_4+a_6)^2,
\]
\[
R_2=a_3 a_6 l_0^2 + a_2 l_0 l_1 l_2 + a_2^2 (l_1+l_2)(l_1+l_3), 
\]
with $l_0=2a_3+a_4$, $l_1=2 a_2+a_5$, $l_2=a_3 + a_6$ and $l_3=a_3 - a_6$.

\medskip
Let us write
\[
R_{113}=\res(p_1',p_3),\quad R_{133}=\res(p_1,p_3').
\]
If $r_1$ (resp. $r_3$) denotes the remainder of the polynomial
division of $p_1$ by $p_1'$ (resp. $p_3$ by $p_3'$), we shall write
\[
\bar R_{113}=\res(r_1,p_3),\quad \bar R_{133}=\res(p_1,r_3).
\]
Note that $D_1,D_3,D_1',\ldots,\bar R_{113},\bar R_{133}$, are 
defined ``for the generic case'', i.e.,  they are obtained 
as expressions on $a_1,\ldots,a_6$ without imposing any condition. 
Some of the expressions are not included in the paper as they are gruesome. 

\medskip

The first result determines the quadratic systems such that $A(\theta)$ has definite sign. 

\begin{bigtheo}\label{teo:A}
The coefficient $A$ has definite sign (and, in consequence, \eqref{system:quadratic} has at most one limit cycle
surrounding the origin) if and only if one of the following conditions holds:
\begin{enumerate}
 \item $p_1$ or $p_3$ is identically null, or, equivalently, one 
 of the following conditions holds:
\begin{enumerate}
 \item $a_6=a_5=3a_3+a_4=a_2=0$,
 \item $a_1a_6-a_2=a_1a_5-a_3+a_4+a_6=a_1(3 a_3+ a_4)-3a_2-a_5=a_1a_2+a_3 = 0$.
\end{enumerate}

\item $p_1$ has degree one, $p_3$ has degree three (i.e.,  $a_6=3a_2+a_5=0$ and $(3a_3+a_4)a_2\neq 0$), 
$R_2=0$, and $a_2^2\leq 4a_3^2 +4 a_1 a_2 a_3$.

\item $p_3$ has degree one, $p_1$ has degree three (i.e.,  $a_2-a_1a_6=2a_3+a_4+a_1(3a_2+a_5)+a_6=0$
and $a_6 \left(3a_2-a_1(3a_3+a_4)+a_5\right) \neq 0$), and one of the following conditions holds:
\begin{enumerate}
 \item $R_2=0$, $D_1\leq  0$, $R_{113}\neq 0$,
 \item $4a_4-9a_6=4a_3+5a_6=9a_2+a_5=9a_1a_6+a_5=8a_1^2-1=0$.
\end{enumerate}
\item $p_1,p_3$ have degree two (i.e.,  $a_2=a_6=0$ and $a_5(a_3-a_1 a_5)\neq 0$), $3a_3+a_4=0$, and
$4 a_3^2 - 4 a_1 a_3 a_5 \geq a_5^2$.
\item $p_1,p_3$ have degree three (i.e., $a_6(a_2-a_1a_6)\neq 0$), $R_2=0$, and 
one of the following conditions holds:
\begin{enumerate}
\item $D_1<0$, $D_3<0$, $(a_3-a_6)\left(a_2^2+(a_4+2a_3)^2\right)\neq 0$,
\item $D_1=0$, $D_3<0$, $D_1' R_{113}\neq 0$,
\item $D_1=D_1'=0$, $D_3<0$,
\item $D_3=0$, $D_1<0$, $D_3' R_{133}\neq 0$,
\item $D_3=D_3'=0$, $D_1<0$,
\item $D_1=D_3=0$, $D_1' D_3' \bar R_{113}\bar R_{133}\neq 0$,
\item $D_1=D_1'=D_3=0$, $\bar R_{133}\neq 0$,
\item $D_1=D_3=D_3'=0$, $\bar R_{113}\neq 0$.
\end{enumerate}
\end{enumerate}

\end{bigtheo}
\begin{rema}
The codimension of the semi-varieties defined by the conditions of Theorem~\ref{teo:A} are
the following (Proposition~\ref{prop:dims}):
\begin{itemize}
 \item $5a)$ has codimension one.
 \item $5b),5d)$ have codimension two.
 \item $2),3a),4)$ have codimension three. 
 \item $1a),1b)$ have codimension four. 
 \item $3b)$ has codimension five.
 \item $5f)$ has codimension two or three.
 \item $5c),5e),5g),5h)$ have codimension of at least two.
\end{itemize}

Note that in case $3b)$ 
the equations already imply $a_2-a_1a_6=2a_3+a_4+a_1(3a_2a_5)+a_6=0$.
\end{rema}

Next, we determine quadratic systems such that $B(t)$ has definite sign.

\begin{bigtheo}\label{teo:B}
The coefficient $B$ has definite sign (indeed, it is identically null) if and only if the parameters $a_1,\ldots,a_6$ belong
to any of the two codimension-four regular varieties defined by the equations
 \begin{equation}\label{p2}
 	a_4+4 a_6=4 a_3+a_4=4 a_2+a_5=a_1=0,
 \end{equation}	
 or
 \begin{equation}\label{p3}
a_6=3a_3+a_4=4 a_2+a_5=3 a_1 a_5+2 a_4=0.
 \end{equation} 
Moreover, \eqref{system:quadratic} has at most one limit cycle surrounding the origin.
\end{bigtheo}

\begin{rema}
The conditions \eqref{p2}, \eqref{p3} in Theorem~\ref{teo:B} imply that $B$ is identically 
null. Therefore, \eqref{eq:Abel} reduces to a Bernoulli equation, and it is possible
to obtain the exact number of limit cycles surrounding the origin (zero or one). 
\end{rema}

The rest of the paper is organized as follows. Section~\ref{sec:2} contains some known results on 
the number of limit cycles of Abel equations. Section~\ref{sec:3} describes the algebraic geometry tools 
that will be required for the proofs of the main results. Section~\ref{sec:main} contains the 
proofs of Theorems~\ref{teo:A} and \ref{teo:B}. Finally, in Appendix~\ref{ap:A} we include
the SINGULAR code for the proofs of Section~\ref{sec:main}.

\section{Abel equations with at most one non-trivial limit cycle}\label{sec:2}
\noindent
In this section we collect known results  about the number of limit cycles of the Abel equation \eqref{eq:Abel} that we will use subsequently.

\begin{prop}[\cite{Pliss, GL}]	
	Assume $A(\theta)$ has definite sign. Then Equation \eqref{eq:Abel} has at most one positive limit cycle.
\end{prop}
\begin{proof}
	From \cite{Pliss}, we have that \eqref{eq:Abel} has at most three limit cycles. Moreover,  notice that $\rho=0$ is always a periodic solution of \eqref{eq:Abel}. 
Since $A(\theta+\pi)=A(\theta)$ and $B(\theta + \pi)=-B(\theta)$, we have that $\rho(\theta)$ is a solution of \eqref{eq:Abel} if and only if $-\rho(\theta +\pi)$ also is. Thus the number of limit cycles is the same in regions $\rho>0$ and $\rho<0$, and consequently Equation \eqref{eq:Abel} has at most one positive limit cycle.
\end{proof}
\begin{prop}
Assume $A(\theta)$ to be identically null. Then Equation \eqref{eq:Abel} has no limit cycle.	
\end{prop}	
\begin{proof}		
	When $A(\theta) \equiv 0$, Equation \eqref{eq:Abel} is the Ricatti equation
	$\rho'= B(\theta)\rho^2+a_1\rho$.  Since $\int_0^{2\pi}B(t)\,dt=0$, when $a_1=0$ it is a centre and if $a_1\neq0$ it  has no limit cycle.
\end{proof}
\begin{prop}
If $B(\theta)$ has definite sign, it is identically null. Moreover, equation\eqref{eq:Abel} has at most one positive limit cycle.
\end{prop}	
\begin{proof}
	Since $B(\theta+\pi)=-B(\theta)$, if $B(\theta)$ has definite sign,  it is necessarily identically null. Then  \eqref{eq:Abel} is the Bernoulli equation $\rho'=A(t)\rho^3+a_1\rho$ which has at most one positive limit cycle.
\end{proof}
\begin{rema}
The criterion $\alpha A+\beta B$ has definite sign for some $\alpha,\beta\in\mathbb{R}$, $\alpha^2+\beta^2\neq 0$, used in \cite{AGG,HZ} to obtain upper bounds for the number of limit cycles in Abel equations is not relevant in this context since if $\alpha A+\beta B$ has definite sign then, by the change of variables $t\to \pi+t$, $\alpha A-\beta B$
has the same definite sign. Therefore $2 \alpha A=(\alpha A+\beta B)+(\alpha A-\beta B)$ has definite sign,
and consequently $A$ has definite sign if $\alpha\neq 0$ and $B(t)\equiv 0$ otherwise. 
\end{rema}  


\section{Algebraic geometry tools}\label{sec:3}

\noindent
In this section, we summarize the computational algebraic geometry results to be used subsequently. In all cases, we will include references to the SINGULAR (\cite{DGPS}) commands necessary to perform the corresponding computation. Those readers interested in considering computational algebraic geometry techniques in more depth are encouraged to consult \cite{CLO} for an introduction, or \cite{Singular-book} for a fuller  development. Furthermore, readers familiar with differential equations will enjoy \cite{Romanovski-Shafer} which includes a comprehensive introduction to the basic generalities of computational algebraic geometry in its first chapter.

Let us consider a system of polynomial equations in $n$ variables $x_1, \ldots, x_n$ with coefficients in a field $\Bbbk$,
\begin{equation}\label{A_ecu1}
\begin{array}{rcl}
f_1(x_1, \ldots, x_n) & = &  0,\\
& \vdots & \\
f_s(x_1, \ldots, x_n) & = & 0.
\end{array}
\end{equation}
Clearly, $(a_1, \ldots, a_n) \in \Bbbk^n$ is a solution of \eqref{A_ecu1} if and only if $$\sum_{i=1}^s g_i(a_1, \ldots, a_n) f_i(a_1, \ldots, a_n) = 0$$ for every $g_i$ in the ring $\Bbbk[x_1, \ldots, x_n]$ of polynomials in $n$ variables  with coefficients in $\Bbbk.$ Thus, the set of solutions of \eqref{A_ecu1} in $\Bbbk^n$ matches the set of zeros in $\Bbbk^n$ of the ideal $\langle f_1, \ldots, f_s \rangle$  of $\Bbbk[x_1, \ldots, x_n]$ generated by $f_1, \ldots, f_s$. The set of zeros of $I  = \langle f_1, \ldots, f_s \rangle$ in $\Bbbk^n$ is called the (affine) variety of $I$ in $\Bbbk^n$.  It is denoted $\mathcal{V}_\Bbbk(I)$, or simply $\mathcal{V}(I)$ when no confusion is possible.

Here, it is convenient to recall that all the ideals of $\Bbbk[x_1, \ldots, x_n]$ are finitely generated by the Hilbert Basis Theorem (see \cite[Theorem 1.3.5]{Singular-book}). Therefore, to study a system of polynomial equations is the same as to study the ideal generated by the polynomials of the system, and vice versa.

Furthermore, since $f(a_1, \ldots, a_n) = 0$ if and only $f^r(a_1, \ldots, a_n) = 0$ for every positive integer $r$, one has that $\mathcal{V}(I) = \mathcal{V}(\sqrt{I})$, where $$\sqrt{I} = \{ f \in \Bbbk[x_1, \ldots, x_n]\ \mid f^r \in I,\ \text{for some}\ r \in \mathbb{Z}_+ \}$$ is the radical of $I$. 

This ideal-variety approach has two immediate advantages. On the one hand, the varieties in $\Bbbk^n$ of the ideals of $\Bbbk[x_1, \ldots, x_n]$ form the closed sets of a topology on $\Bbbk^n$ called the Zariski topology of $\Bbbk^n$ (see \cite[Lemma A.2.4]{Singular-book}).  And on the other, there exists of a kind of factorization theory for ideals of $\Bbbk[x_1, \ldots, x_n]$ in which the intersection of ideals plays the role of the product: 
the so-called primary decomposition theory that we shall outline in the following.

Observe that because of the well-known property $$\mathcal{V}(J_1 \cap J_2) = \mathcal{V}(J_1) \cup \mathcal{V}(J_2),$$ for $J_i,\ i = 1,2,$ ideals of $\Bbbk[x_1, \ldots, x_n]$ (see \cite[Lemma A.2.3, part (2)]{Singular-book}), a decomposition of the ideal defined by the polynomials in \eqref{A_ecu1} as an intersection of ``simpler ideals'' will mean splitting the system \ref{A_ecu1} into several easier-to-solve systems, hopefully!

Depending on the purpose, some systems of generators of a polynomial ideal are better than others. For example, minimal systems of generators (i.e., systems of generators such that no generator is an algebraic combination of the others) are preferred for a concise description of the variety.  But Gr\"obner bases, which are far from being minimal in the above sense, are special systems of generators with good computational properties.  Given a system of generators of an ideal $I$ of $\Bbbk[x_1, \ldots, x_n]$, one can compute a minimal system of generators or a Gr\"obner basis of $I$ by using the SINGULAR commands \texttt{mres(I,1)[1]} or \texttt{std(I)}, respectively.

The original aim of the Gr\"obner bases methods was to compute the remainder of a polynomial under division by a polynomial ideal, something that can be done with the command \texttt{reduce} in SINGULAR. Nowadays, Gr\"obner bases are used for more sophisticated tasks.  Computing the dimension of a variety or eliminating variables are just two classic examples.

Given a system of generators of an ideal $I$ of $\Bbbk[x_1, \ldots, x_n]$, the problem of the computation of the dimension of $\mathcal{V}(I)$ (equivalently, the Krull dimension of $\Bbbk[x_1, \ldots, x_n]/I$) may be reduced to a pure combinatorial problem after the computation of one (any) Gr\"obner basis of $I$ (see \cite[Chapter 9]{CLO}). The SINGULAR command \texttt{dim(std(I))} will compute the dimension of $\mathcal{V}_\mathbb{C}(I)$ for us. The precise notion of dimension will be defined at the end of this section. On other hand, the problem of the elimination of a variable, say $x_n$, from the ideal $I$, consists of determining a system of generators of $I \cap \Bbbk[x_1, \ldots, x_{n-1}]$. This can be easily computed from a Gr\"obner basis of $I$ with respect to a suitable well-ordering of the monomials in $\Bbbk[x_1, \ldots, x_n]$. Geometrically, the elimination of variables has the following meaning:

\begin{prop} 
Let $\Bbbk$ be algebraically closed, and let $I$ be an ideal of $\Bbbk[x_1, \ldots, x_n]$. If $\pi : \Bbbk^n \to \Bbbk^{n-1}$ is the projection map that sends $(a_1, \ldots, a_n)$ to $(a_1, \ldots, a_{n-1})$ then the Zariski closure of $\pi(\mathcal{V}(I))$ in $ \Bbbk^{n-1}$ is equal to $\mathcal{V}(I \cap \Bbbk[x_1, \ldots, x_{n-1}])$.
\end{prop}

\begin{proof}
See \cite[Theorem 3, Section 3.2]{CLO}.
\end{proof}

The elimination of variables is computed in SINGULAR with the command \texttt{eliminate}.

Let us now briefly summarize  the primary decomposition process for ideals of $\Bbbk[x_1, \ldots, x_n]$. To do so, we shall first introduce the quotient operation and its most elementary properties.

\begin{defi}
Let $I$ and $J$ be ideals of $\Bbbk[x_1, \ldots, x_n]$. The quotient of $I$ and $J$ is the ideal $(I:J)$  of $\Bbbk[x_1, \ldots, x_n]$ defined as follows: $$(I:J) = \{ g \in \Bbbk[x_1, \ldots, x_n]\ \mid\ g f \subseteq I,\ \text{for every}\ f \in J \}.$$
\end{defi}

It is not difficult to see that $I \subseteq (I:J)$ and $\big((I:J) : J\big) = (I:J^2)$. Then, we have a chain of ideals $I \subseteq (I:J) \subseteq \ldots \subseteq (I:J^r) \subseteq \ldots$ that necessarily stabilizes by the Noetherian property of $\Bbbk[x_1, \ldots, x_n]$. If $N$ is the smallest integer for which the above chain stabilizes, then the ideal $(I:J^N)$ is called the saturation of $I$ by $J$ and is usually denoted $(I:J^\infty)$.

Both quotient and saturation can be computed using the SINGULAR commands \texttt{quotient} and \texttt{sat}, respectively (the latter from the \texttt{elim} library).

\begin{rema}\label{rema:membership}
Observe that an elementary necessary and sufficient condition for $J \subseteq I$ is $(I:J) = \langle 1 \rangle$.
Moreover, one has that $f \in \sqrt{I}$ if and only if $(I:\langle f \rangle^\infty) = \langle 1 \rangle$. So, the radical membership problem can be computationally solved by computing the saturation of $I$ by $\langle f \rangle$.
\end{rema}

Geometrically, when $\Bbbk$ is algebraically closed, the quotient and the saturation of $I$ by $J$ have the same behaviour which is nothing but the Zariski closure of the difference of varieties.  In particular, the following holds: $$\mathcal{V}(I:J) = \overline{\mathcal{V}(I) \setminus \mathcal{V}(J)} = \overline{\mathcal{V}(I) \setminus \mathcal{V}(J^r)} = \mathcal{V}(I:J^r),$$ for every positive integer $r$ (see \cite[Theorem 7, section 4.4]{CLO}).

The next result represents a first step for the decomposition of an ideal of $\Bbbk[x_1, \ldots, x_n]$:

\begin{lema}{\rm \textbf{(Splitting tool).}}\label{LemaPD1}
Let $I$ be an ideal of $\Bbbk[x_1, \ldots, x_n]$, and let $g \in \Bbbk[x_1, \ldots, x_n]$. If $N$ is the smallest integer such that $(I:\langle g \rangle^\infty) = (I:\langle g^N \rangle)$, then $$I = \big(I:\langle g \rangle^\infty\big) \cap \big(I+\langle g^N \rangle \big).$$
\end{lema}

\begin{proof}
See \cite[Lemma 3.3.6]{Singular-book}.
\end{proof}

When $\Bbbk$ is algebraically closed, an immediate consequence of the splitting tool is the formula $$\mathcal{V}(I) = \mathcal{V}(I:f^\infty) \cup \mathcal{V}(I+\langle f \rangle) = \overline{\mathcal{V}(I) \setminus \mathcal{V}(f)} \cup \big(\mathcal{V}(I) \cap \mathcal{V}(f) \big)$$ where the varieties in the union on the right-hand side can be carefully interpreted as the solutions of the system associated with $I$ by imposing the conditions $f(x_1, \ldots, x_n)$ different from or equal to zero, respectively.

At this point, we are in a position to clarify what ``simpler ideals'' means in the context of primary decomposition theory.

\begin{defi}
An ideal $P$ of $\Bbbk[x_1, \ldots, x_n]$ is said to be prime if $f g \in P$ and $g \not\in P$ implies $f \in P$. An ideal $Q$ of $\Bbbk[x_1, \ldots, x_n]$ is said to be primary if $f g \in Q$ and $g \not\in Q$ implies $f \in \sqrt{Q}$.
\end{defi}

Notice that every \emph{prime ideal $P$ is primary}: indeed, if $P$ is prime, $\sqrt{P} = P$. Moreover, one can easily check that the \emph{radical of a primary ideal is prime}.
Here, it is important to emphasize that, if $\Bbbk$ is algebraically closed, then $P$ is a prime ideal of $\Bbbk[x_1, \ldots, x_n]$ if and only if $\mathcal{V}(P)$ is Zariski irreducible (see \cite[Corollary 4, Section 4.5]{CLO}). So, in this case, the variety of a primary ideal is a Zariski irreducible subset of $\Bbbk^n$.

%

\begin{theo}\label{ThPD1}
Let $I$ be an ideal of $\Bbbk[x_1, \ldots, x_n]$. If $I \neq \langle 1 \rangle,$ there exists a decomposition of $I$ as the intersection of finitely many primary ideals.
\end{theo}

\begin{proof}
If $I$ is primary, there is nothing to prove. Otherwise, there exists $g \not\in \sqrt{I}$ such that $(I:g^\infty) \supsetneq I$. Thus, by Lemma \ref{LemaPD1}, $I$ decomposes as $(I:g^\infty) \cap (I + \langle g^N \rangle)$. Both ideals strictly contain $I$. If they are primary, we are done. Otherwise, we can repeat the same argument with $(I:g^\infty)$ and $(I + \langle g^N \rangle)$, and so on and so forth. In so far as this process cannot continue indefinitely because of the Noetherian property of $\Bbbk[x_1, \ldots, x_n]$, our claim follows.
\end{proof}

A decomposition of $I$ into primary ideals, $I = Q_1 \cap \ldots \cap Q_r$, is called a primary decomposition of $I$. Since $\sqrt{I} = \sqrt{Q_1} \cap \ldots \cap \sqrt{Q_s}$, by removing redundancies if necessary, we obtain finitely many prime ideals, $P_1, \ldots, P_t$, not contained one in another, such that $$\mathcal{V}(I) = \mathcal{V}(P_1) \cup \ldots \cup \mathcal{V}(P_t) .$$ Therefore, when $\Bbbk$ is algebraically closed, a primary decomposition of an ideal $I$ yields a decomposition of $\mathcal{V}(I)$ into Zariski irreducible varieties. In general, the prime ideals defining these varieties do not depend on the decomposition, and are called minimal associated primes of $I$ (\cite[Theorem 4.1.5]{Singular-book}).

\begin{rema}\label{rem:MinAss}
Let $\{P_1, \ldots, P_t\}$ be the set of minimal associated primes of an ideal $I$ of $\Bbbk[x_1, \ldots, x_n]$. If $P'$ is a prime ideal such that $I \subseteq P' \subseteq P_j$ for some $j$, then $P' = P_j$. Indeed, it suffices to note that $\sqrt{I} = P_1 \cap \ldots \cap P_t \subseteq P' \subseteq P_j$ implies $P_i \subseteq P' \subseteq P_j$ for some $i$, and that necessarily $i = j$. Therefore, the minimal associated primes of $I$ are the ``smallest'' prime ideals containing $I$.
\end{rema}

In conclusion, there exists a computational method to write the set of solutions of a system of polynomial equations in several variables as the union of the solution of finitely many systems. Moreover, if $\Bbbk$ is algebraically closed, the varieties associated with those systems are Zariski irreducible.

The minimal associated primes of an ideal of $\Bbbk[x_1, \ldots, x_n]$ can be computed by using the SINGULAR command \texttt{minAssGTZ} (library \texttt{primary}).

We end this section by defining the notion of dimension of an algebraic variety.

\begin{defi}
Let $I$ be an ideal of $\Bbbk[x_1, \ldots, x_n]$. The dimension of $I,\ \dim(I)$, is the supremum of the lengths of all chains of prime ideals in $\Bbbk[x_1, \ldots, x_n]/I$.
\end{defi}

Equivalently, $\dim(I)$ is supremum of the lengths of all chains of prime ideals in $\Bbbk[x_1, \ldots, x_n]$ containing $I$ (because of the well-known correspondence between ideals of the quotient $A/I$ and ideals of $A$ containing $I$). Observe that 
$$\dim(I) = \max \big\{\dim(P) \mid P\ \text{is a minimal associated prime of}\ I\big\}$$ by Remark \ref{rem:MinAss}. 

This notion of dimension does not depend on the base field $\Bbbk$ in the sense that if $\Bbbk \hookrightarrow \mathbb{K}$ is an extension of $\Bbbk$, then the dimension of $I$ is the same regardless of whether $I$ is an ideal of $\Bbbk[x_1, \ldots, x_n]$ or an ideal of $\mathbb{K}[x_1, \ldots, x_n]$ (see \cite[Theorem 3.5.1]{Singular-book}).

Since the dimension of $\mathcal{V}(I)$ is the supremum of the lengths of the chains of its closed irreducible sets, when $\Bbbk$ is algebraically closed, the dimension of $\mathcal{V}(I)$ is the maximum of $\dim(P)$ where $P$ is any minimal associated prime of $I$. 

The next result is a particular version of the General Jacobian criterion (see \cite[Theorem 5.7.1]{Singular-book}). 

\begin{theo}\label{theo:dimprimes}
Let $I = \langle f_1, \ldots, f_m \rangle \subset \Bbbk[x_1, \ldots, x_n]$ be an ideal and $P$ a minimal associated prime of $I$. If $\mathbf{a} = (a_1, \ldots, a_n) \in \mathcal{V}(P) \subseteq \Bbbk^n$, then \begin{equation}\label{ecu:jacob}\mathrm{rank}\left(\frac{\partial f_i}{\partial x_j} (\mathbf{a})\right) \leq n-\dim(P), \end{equation} and $\mathbf{a} = (a_1, \ldots, a_n)$ is a regular point of $\mathcal{V}(I)$ if and only if the equality holds.
\end{theo}

\begin{proof}
This theorem is nothing but \cite[Theorem 5.7.1]{Singular-book}, from taking into account that $n-\dim(P)$ is the height of $P,\ \mathrm{ht}(P),$ by \cite[Theorem 3.5.1(4)]{Singular-book} and the definition of regular point given in \cite[Definition A.8.7]{Singular-book}.
\end{proof}

The left hand side in \eqref{ecu:jacob} can be computed in SINGULAR with the following command \texttt{rank(reduce(jacob(I),std(m\_a)))}, where $\mathfrak{m}_\mathbf{a}$ is the maximal ideal associated with $\mathbf{a}$, i.e., $\mathfrak{m}_\mathbf{a} = \langle x_1 - a_1, \ldots, x_n - a_n \rangle$.

\section{Proof of the main results}\label{sec:main}

\noindent
In this section, we shall prove Theorem~\ref{teo:A} and Theorem~\ref{teo:B}. 

A first consideration is that
the functions $A$ and $B$ are homogeneous trigonometric polynomials of degrees $6$ and $3$, respectively. 
Since $\sin(\theta)=-\sin(\theta+\pi)$ and $\cos(\theta)=-\cos(\theta+\pi)$, then for all $\theta\in (-\pi/2,\pi/2]$
\[
A(\theta)=A(\theta+\pi),\quad B(\theta)=-B(\theta+\pi).
\]
In particular, $B$ has definite sign if and only if $B(\theta)\equiv 0$ for all $\theta\in(-\pi/2,\pi/2]$,
and $A$ has definite sign if and only if $A(\theta)\geq 0$ for all $\theta\in(-\pi/2,\pi/2]$, or
$A(\theta)\leq 0$ for all $\theta\in(-\pi/2,\pi/2]$.

\medskip

By the changes of variables $x=\tan(\theta)$, we obtain that $A$ has definite sign if and only if
the rational function
\[
A(\atan(x))=\frac{p_1(x)\Big(a_1p_1(x)-p_2(x)\Big)}{\left(1+x^2\right)^3} 
\]
has definite sign, where (we recall)
\[\begin{split}
p_1(x)&=a_2+(3 a_3+a_4) x-(3 a_2+a_5) x^2-a_6 x^3,\\
p_2(x)&=- a_3 + (3 a_2 + a_5) x + (2 a_3 + a_4 + a_6) x^2 - a_2  x^3,
\end{split}\]
or equivalently, that $p_1(x) \Big(a_1p_1(x)-p_2(x)\Big)$ has definite sign.

\medskip

Analogously, by the change of variable $x=\tan(\theta)$, $B$ is identically null
if and only if 
\[
B(\atan(x))=\frac{q(x)}{\left(1+x^2\right)^{3/2}}\equiv 0,\quad \text{for all }x\in\mathbb{R},
\]
where

\[
\begin{split}
q(x)=&-\Big(2a_1a_2+4a_3+a_4\Big)+\Big(12a_2+3a_5-a_1(6a_3+2a_4)\Big)x\\
&+\Big(8a_3+3a_4+4a_6+a_1(6a_2+2a_5)\Big)x^2-\Big(4a_2+a_5-2a_1a_6\Big)x^3.
\end{split}
\]
Again, that is equivalent to $q(x)\equiv 0$.

\subsection{Proof of Theorem \ref{teo:A}}
We divide the proof of Theorem~\ref{teo:A} into several propositions.
A first comment is that if $p_1(x)\equiv 0$ or $p_3(x):=a_1 p_1(x)-p_2(x)\equiv 0$ then 
$A(\theta)\equiv 0$. 
In Proposition~\ref{propo:Anull} we characterize when one of the polynomials $p_1$, $p_3$
is identically null. 
Next, we distinguish cases in terms of the minimum of the degrees
of $p_1$ and $p_3$. When this minimum is zero, Theorem~\ref{teo:A}
is proved in Proposition~\ref{prop:Ag0}; when it is one, in Proposition~\ref{prop:Ag1}; when it is two, in Proposition~\ref{prop:Ag2}; and when it is three, in Proposition~\ref{prop:g33r1s}.


\begin{prop}\label{propo:Anull}
The polynomial $p_1p_3$ is identically null if and only if 
\begin{equation}\label{eq:1}
a_6=a_5=3a_3+a_4=a_2=0, 
\end{equation}
or
\begin{equation}\label{eq:2}
\begin{split}
a_1a_6-a_2=a_1a_5-a_3+a_4+a_6 =  0, \\
a_1(3 a_3+ a_4)-3a_2-a_5=a_1a_2+a_3 = 0.
\end{split}
\end{equation}

\end{prop}
\begin{proof}
If suffices to consider the ideals generated by the coefficients of 
the polynomials, and then, for each of these ideals, compute a minimal system of generators. (See Appendix~\ref{ap:A}).
\end{proof}

In the following, we assume that neither of $p_1$, $p_3$ is identically null. 
In consequence, $p_1p_3$ has definite sign if and only if
the odd-multiplicity real roots of $p_1$ and $p_3$ coincide. 
We shall distinguish several cases depending on the minimum degree of $p_1$ and $p_3$.

\medskip

If the minimum degree of $p_1$ and $p_3$ is zero (and neither of $p_1,p_2,p_3$ is identically null), then
$A$ does not have definite sign.

\begin{prop}\label{prop:Ag0}
If $p_1$ and $p_3$ are not identically null and $p_1$ or $p_3$ is constant, then the odd-multiplicity real roots of 
$p_1$ and $p_3$ do not coincide.
\end{prop}
\begin{proof}
Assume $p_1$ is constant, i.e.,  $p_1(x)=a_2$. If the odd-multiplicity roots of $p_1,p_3$ coincide, then
$p_3$ has even degree. Hence $a_2=0$, in contradiction with $p_1$ not being null.

Conversely, if $p_3$ is constant and not null, then $p_3(x)= a_1a_2+a_3$. Arguing as above, $p_1$ has an even degree,
so $a_6=0$. Moreover, since $p_3(x)$ is constant, $a_2=0$ in particular, and
\[p_1(x)= x(3a_3+a_4-a_5x).\]
I.e., $x=0$ is a root of $p_1$. If it is a simple root, it should be a root of $p_3$, in contradiction with $p_3$ being constant, so that $3a_3+a_4=0$. But in this case, 
\[
p_3(x)=a_3 - a_5 x + \left(-2 a_3 - a_4 - a_1 a_5 \right) x^2.
\]
In particular, $a_5=0$, so $p_1(x)\equiv 0$, and with this contradiction we conclude the proof. 
\end{proof}


Next, we consider that one of $p_1,p_3$ has degree one, and the other
has an equal or greater degree.

\begin{prop}\label{prop:Ag1}
Assume that the minimum of the degrees of $p_1$ and $p_3$ is one. 
Then the odd-multiplicity real-roots of $p_1,p_3$ coincide
if and only if 
\begin{equation}\label{prop:Ag1-1}
a_6=3a_2+a_5=R_2=0,\ a_2^2\leq 4a_3^2 +4 a_1 a_2 a_3,\ a_2(3a_3+a_4)\neq 0,
\end{equation}
or
\begin{equation}\label{prop:Ag1-2}
\begin{split}
& a_2-a_1a_6=2a_3+a_4+a_1(3a_2+a_5)+a_6=R_2=0,\\
& D_1\leq 0,\ a_6\left(3a_2-a_1(3a_3+a_4)+a_5\right)R_{113}\neq 0.	                                    
\end{split}
\end{equation}
or 
\begin{equation} \label{prop:Ag1-3}
4a_4-9a_6=4a_3+5a_6=9a_2+a_5=9a_1a_6+a_5=8a_1^2-1=0,\ a_6\neq0.
\end{equation}
\end{prop}
\begin{proof}
Assume the odd-multiplicity real-roots of $p_1,p_3$ coincide.
Then the possible degrees of $p_1,p_3$ are one or three. 
The polynomials $p_1$ and $p_3$ can not be simultaneously linear, since in this case $p_1(x)=a_3x$ and $p_3(x)= a_3+a_1a_3x$.

\textsc{Case 1.}
Assume that $p_1$ has degree one and $p_3$ has degree three. 
Then $a_6=0$, $3a_2+a_5=0$, $3a_3+a_4\neq0$, and $a_2\neq 0$. 
Moreover, 
\begin{equation*}
\begin{split}
p_1(x)&= a_2+ (3 a_3 + a_4) x, \\
p_3(x)&= a_1a_2+a_3+a_1(3a_3+a_4)x-(2a_3+a_4)x^2+a_2x^3.	
\end{split}
\end{equation*}
Assume that the odd-multiplicity roots of $p_1$ and $p_3$ coincide.
The root of $p_1$ is $x_1=-a_2/(3a_3+a_4)$. 
Then
\[
p_3(x_1)=-\frac{(-a_2^2+a_3(3a_3+a_4))(a_2^2+(3a_3+a_4)^2)}{(3a_3+a_4)^3}=0.
\]
Hence
\[
0=a_2^2-a_3(3a_3+a_4)=\frac{R_2}{a_2^2}.
\]
In consequence $a_3\neq 0$. Replacing $a_4$ by $\frac{a_2^2-3a_3^2}{a_3}$, we obtain
\[
p_1(x)=\frac{a_2 (a_3 + a_2 x)}{a_3},\quad p_3(x)=\frac{(a_3 + a_2 x) (a_1 a_2 + a_3 - a_2 x + a_3 x^2)}{a_3}.
\]
The odd-multiplicity roots of $p_1$ and $p_3$ coincide if and only if
$a_1 a_2 + a_3 - a_2 x + a_3 x^2$ has no simple roots, i.e.,   
\[
a_2^2 - 4 a_1 a_2 a_3 - 4 a_3^2\leq 0.
\]

The converse is obvious.
\medskip

\textsc{Case 2}. Assume that $p_3$ has degree one and $p_1$ has degree three, or equivalently 
\begin{equation*} \begin{split}
		& a_2=a_1a_6,\ 2a_3+a_4+a_1(3a_2+a_5)+a_6=0,\\
		& a_6\neq0,\ 3a_2-a_1(3a_3+a_4)+a_5 \neq0.
\end{split}
\end{equation*}
Assume that the odd-multiplicity real roots of $p_1$ and $p_3$ coincide.
From $a_2= a_1 a_6$, $a_4= -2 a_3 - a_1 a_5 - a_6 - 3 a_1^2 a_6$, we obtain 
\[
p_3(x)=a_3 + a_1^2 a_6 - (a_5 + a_1 (-a_3 + a_1 a_5 + 4 a_6 + 3 a_1^2 a_6)) x.
\]
where $a_5 + a_1 (-a_3 + a_1 a_5 + 4 a_6 + 3 a_1^2 a_6)\neq 0$. Therefore, $p_3$ has the unique root 
\[
x_0=\frac{a_3 + a_1^2 a_6}{a_5 + a_1 \left(-a_3  + a_1 a_5 + 4  a_6 + 3 a_1^2 a_6\right)}.\]
As $p_1$ and $p_3$ have the same odd-multiplicity real roots, $x_0$ must be a root of $p_1$.
Substituting, one has
\[
p_1(x_0)=\frac{-R_2 \left( (a_5 + 4 a_1 a_6)^2 + (a_3 - a_1 a_5 - 3 a_1^2 a_6)^2 \right)}
{a_6^2 \left(a_5 + a_1 \left(-a_3  + a_1 a_5 + 4  a_6 + 3 a_1^2 a_6\right)\right)^3}.
\]
Since $a_5 - a_1 (a_3 - a_1 a_5 - 4 a_6 - 3 a_1^2 a_6)\neq 0$, then 
$(a_5 + 4 a_1 a_6)^2 + (a_3 - a_1 a_5 - 3 a_1^2 a_6)^2>0$. Therefore
$x_0$ is a root of $p_1$ if and only if $R_2=0$. 

\medskip

If $D_1<0$, we shall prove that $R_{113}\neq 0$, so that \eqref{prop:Ag1-2} holds.
Assume by contradiction that $R_{113}=0$. Consider the ideal generated by $D_1+x^2$ (which implies $D_1<0$ if $x\neq 0$), $a_2-a_1a_6$, $2a_3+a_4+a_1(3a_2+a_5)+a_6$, and $R_{113}$. This ideal has three associated primes (see Appendix~\ref{ap:A} - the computations take some time in this case). The first one contains the polynomial $x$, so that it corresponds to $D_1=0$. The second contains the polynomial $3a_2-a_1(3a_3+a_4)+a_5$.
The third contains $1+a_1^2$ so that it has no real points. Therefore, the variety 
of the ideal is contained in $3a_2-a_1(3a_3+a_4)+a_5=0$.  But $3a_2-a_1(3a_3+a_4)+a_5\neq 0$ by hypothesis. 
This contradiction proves that $R_{113}\neq 0$.

If $D_1=0$, the multiplicity of $x_0$ as a root of $p_1$ must be one or three. 
The multiplicity is two or more if and only if $p_1'(x_0)=0$, but
\[
R_{113}=9 a_6^2 \left(a_5 + a_1 \left(-a_3  + a_1 a_5 + 4  a_6 + 3 a_1^2 a_6\right)\right)^2 p_1'(x_0). 
\]
I.e., the multiplicity is one if and only if $R_{113} \neq 0$.
Finally, if the multiplicity is three, then $p_1(x)=-a_6(x-a)^3$ for a certain $a$.
We consider the ideal generated by $R_2$, $a_2-a_1a_6$, $2a_3+a_4+a_1(3a_2+a_5)+a_6$, and the coefficients of $p_1(x)+a_6(x-a)^3$. Eliminating $a$, and computing the minimal associated primes, we obtain 
three ideals. The first one contains $1+a_1^2$ so that it has no real points in its variety. The second 
contains the polynomial $a_6$, and, since by hypothesis $a_6\neq 0$, it has no real points in its variety.
The third is
\[
\begin{split}
&8 a_5^2-81 a_6^2=
4 a_4-9 a_6=
9 a_1 a_6+a_5=
8 a_1 a_5+9 a_6=
8 a_1^2-1=0,\\
&3 a_1^2 a_6+ a_1 a_5+2 a_3+ a_4+ a_6=-a_1 a_6+a_2=0.
\end{split}
\]
Computing a minimal system of generators, 
we obtain 
\begin{equation}\label{eq:1m3}
4a_4-9a_6=4a_3+5a_6=9a_2+a_5=9a_1a_6+a_5=8a_1^2-1=0. 
\end{equation}
To conclude, note that if \eqref{eq:1m3} holds then 
\[
p_1(x)=-a_6\left(x\pm \frac{1}{\sqrt{2}} \right)^3,\quad p_3(x)=\frac{9\sqrt{2}}{8}a_6\left(x\pm \frac{1}{\sqrt{2}} \right).
\]
\end{proof}

Now, we consider that either $p_1$ or $p_3$ has degree two (and the other degree is two or more).

\begin{prop}\label{prop:Ag2}
Assume that the minimum of the degrees of $p_1$ and $p_3$ is two. Then the odd-multiplicity real roots of $p_1,p_3$  
coincide if and only if 
\begin{equation}\label{eq:6}
a_2=a_6=3a_3+a_4=0,\ 4 a_3^2 - 4 a_1 a_3 a_5 \geq a_5^2>0,\ a_3-a_1 a_5\neq 0.
\end{equation}
\end{prop}

\begin{proof}
Assume that the minimum of the degrees of $p_1$ and $p_3$ is two and the real odd-multiplicity roots of $p_1,p_3$ coincide. Note that this implies that they are both of degree two. I.e., $a_6=a_2=0$, $a_5\neq 0$, $2a_3+a_4+a_1a_5\neq0$, and
\begin{equation*}
\begin{split}
    p_1(x)&= (3 a_3 + a_4) x - a_5 x^2,\\
    p_3(x)&= a_3 +(a_1(3a_3+a_4)-a_5) x - (2 a_3 + a_4 + a_1a_5) x^2.    
\end{split}
\end{equation*}    

The roots of $p_1$ are then $x_1=0$ and $x_2=(3 a_3 + a_4)/a_5$.

Assume that $3 a_3 + a_4\neq 0$. As the simple real roots of $p_1$ must be roots of $p_3$, we have that $p_3(0)=0$ 
which implies $a_3=0$. Moreover, evaluating $p_3$ at $x_2$, we obtain
\[
p_3\left(\frac{3a_3+a_4}{a_5}\right)=-\frac{a_4 (a_4^2 + a_5^2)}{a_5^2}=0.
\]
I.e., $a_4=0$. But this is contradictory with $3 a_3 + a_4\neq 0$.

\medskip

If $3a_3+a_4=0$ then $p_1(x)=-a_5x^2$ has no odd-multiplicity real roots. The discriminant of $p_3$,
replacing $a_4$ by $-3a_3$, is
\[
\disc(p_3)=-4 a_3^2 + 4 a_1 a_3 a_5 + a_5^2,
\]
so that $p_3$ has no simple real roots if and only if $4 a_3^2 - 4 a_1 a_3 a_5 - a_5^2\geq 0$.
Finally, note that if $3a_3+a_4=0$ then the condition $2a_3+a_4+a_1a_5\neq0$ is equivalent to $a_3-a_1a_5\neq0$. 

Conversely, assume that \eqref{eq:6} holds.  Then $p_1(x)=-a_5 x^2$ and $p_3(x)= a_3-a_5x+ (a_3-a_1a_5)x^2$. Since $\disc (p_3)<0$, both $p_1$ and $p_3$ have no odd real roots.
\end{proof}

In the remainder of this subsection, we shall consider that both $p_1$ and $p_3$ have degree three. 
In this case, the number of real odd-multiplicity roots is given by the discriminant, being
three if the discriminant is strictly positive and one if the discriminant is negative.
Note that if $D_1\leq 0$ and $D_3>0$, or $D_1> 0$ and $D_3\leq 0$, then 
the odd-multiplicity roots of $p_1,p_3$ do not coincide
since one has three simple roots and the other has one root with odd-multiplicity. Consequently,
we only need to consider the cases $D_1,D_3>0$ or $D_1,D_3\leq 0$.

Firstly, we consider the case when $p_1,p_3$ have three simple roots, for which we prove that the real roots can not coincide. The following result is a little more
general since we do not impose the condition that the real roots be simple. It will be used in proving other cases.

\begin{prop}\label{prop:g33r3}
If $p_1,p_3$ have three real roots then the roots do not coincide (with multiplicity).
\end{prop}
\begin{proof}
The polynomials $p_1,p_3$ have three real roots if 
and only if $a_6, a_2-a_1a_6\neq 0$ and their discriminants are positive.

The three real roots of $p_1,p_3$ coincide (with multiplicity) if and only if there exists $\lambda\in\mathbb{R}$ such that
$p_1(x)=\lambda p_3(x)$. Equating the coefficients of the leading term, one obtains
\[
\lambda=\frac{a_6}{-a_2 + a_1 a_6}.
\]
Replacing $\lambda$ in the rest of the equations yields the system (we have multiplied by $a_2-a_1a_6\neq 0$)
\[
a_2^2 + a_3 a_6=a_2(3 a_3 + a_4 - 3 a_6) - a_5 a_6=a_2(3 a_2 + a_5) + a_6(2 a_3 + a_4 + a6)=0.
\]
Solving this, one obtains (note that it is a staggered solution)
\[
a_3=\frac{-a_2^2}{a_6},\quad
a_5= \frac{a_2(a_4 a_6 + 3 a_2 - 3 a_2^2)}{a_6^2},\quad  
a_4=\frac{3 a_2^2 - a_6^2}{a_6}.
\]
Substituting in $D_1$ gives $D_1=-4(a_2^2+a_6^2)^2< 0$, in contradiction with $p_1$ having three real roots.
\end{proof}

Recall that $\res(p_1,p_3)$ factorizes as 
the product of two polynomials, $R_1,R_2$. We shall prove that if $p_1,p_3$ have a 
real root in common then $R_2$ must vanish.

\begin{lema}\label{lema:R1}
Assume $a_2-a_1a_6,a_6\neq 0$.
$p_1,p_3$ have a real root in common if and only if $a=(a_1,\ldots a_6)\in \Var(R_2)$.
\end{lema}
\begin{proof}
If $p_1,p_3$ have a real root in common, then $\res(p_1,p_3)=R_1 R_2=0$. Hence $R_1=0$ or $R_2=0$.	
Assume that $R_1=(4a_2+a_5)^2+(3a_3+a_4+a_6)^2=0$, i.e., $a_5=-4a_2$ and $a_6=-3a_3-a_4$. Then
\begin{equation*}\begin{split}
p_1(x)&=\left(a_2 + (3 a_3 + a_4) x\right) (1 + x^2),\\
p_3(x)&=\left(a_1 a_2 + a_3 + (a_2 + 3 a_1 a_3 + a_1 a_4) x\right) (1 + x^2).
\end{split}
\end{equation*}
Therefore, $p_1,p_3$ have a real root in common if and only if $a_2^2 - 3 a_3^2 - a_3 a_4=0$. 
Since $R_1=(4a_2+a_5)^2+(3a_3+a_4+a_6)^2=0$ then
\[
R_2=(a_2^2 - 3 a_3^2 - a_3 a_4) \left(4 a_2^2 + (2 a_3 + a_4)^2\right).
\]
Thus, the real root coincide if and only if $R_2=0$. 
\end{proof}

Next, we study the singular points of the variety defined by $R_2$. We shall show that they are the intersection
of the variety with the hyperplane $a_3=a_6$. Moreover, in the intersection, the odd-multiplicity real 
roots of $p_1$ and $p_3$ do not coincide.

\begin{lema}\label{lema:sing}
The point $a=(a_1,a_2,\ldots,a_6)\in \Var(R_2)$ is singular if and only if $a_3=a_6$ or $a_2=2a_3+a_4=0$.

Moreover, if $a\in \Var(R_2)$ is singular, then the real odd-multiplicity roots of $p_1,p_3$ do not coincide.
\end{lema}
\begin{proof}
The variety of singular points of $\Var(R_2)$ is defined by 
$\Var(\langle R_2,\nabla R_2\rangle)$. It has two minimal
associated prime ideals (see the SINGULAR code in Appendix~\ref{ap:A}), 
\begin{equation}\label{eq:singular_points}
\langle 2a_2^2+a_2a_5+a_4a_6+2a_6^2,a_3-a_6\rangle\quad\text{and}\quad \langle 2a_3+a_4, a_2\rangle.
\end{equation}
If $a_3=a_6$, then 
$
R_2=(2 a_2^2 + a_2 a_5 + a_4 a_6 + 2 a_6^2)^2
$. 
Hence, 
\begin{equation*}
R_2=0, a_3=a_6\quad \text{ if and only if }\quad R_2=0, \nabla R_2=0.	
\end{equation*}
Let $a\in \Var(\langle R_2,a_3-a_6\rangle)$. Then, 
parametrizing the variety by $a_1,a_2,a_3,a_4$, we obtain
\[
p_1(x)=\frac{(a_3 x + a_2) (a_2 + (2 a_3 + a_4) x - a_2 x^2)}{a_2},
\]
\[
p_3(x)=\frac{(a_2 + (2 a_3 + a_4) x - a_2 x^2) (a_3 - a_2 x + a_1 (a_2 + a_3 x))}{a_2}.
\]
I.e., $p_1,p_3$ have three real roots (as the quadratic factor has positive discriminant),
and by Proposition~\ref{prop:g33r3} they do not coincide.

\medskip

Finally, let $a\in \Var(\langle a_2,2a_3+a_4\rangle)$. Then
\[
p_1(x)=x (a_3 - a_5 x - a_6 x^2),\quad p_3(x)=(1 + a_1 x) (a_3 - a_5 x - a_6 x^2).
\]
Since $x=0$ has different parity as root of $p_1$ than it does as root of $p_3$, 
they do not have the same odd-multiplicity  real roots.
\end{proof}

The next proposition considers the case of $p_1$ and $p_3$ having a unique simple real solution.

\begin{prop}
Assume $p_1$ and $p_3$ have one simple root and two complex conjugate roots.
Then $p_1$ and $p_3$ have the same odd-multiplicity real root if and only if
\begin{equation}\label{eq:7}
\begin{split}
&R_2=0,\ D_1<0,\ D_3<0,\ \\
&a_6\neq 0,\ a_2-a_1a_6\neq 0,\ a_3\neq a_6,\ a_2^2+(a_4+2a_3)^2\neq 0. 
\end{split}
\end{equation}
\end{prop}

\begin{proof} If $p_1$ and $p_3$ have the same real root then $R=0$ and, by Lemma~\ref{lema:R1}, $R_2=0$. Moreover, applying Lemma~\ref{lema:sing}, $a_3\neq a_6$, and either $a_2\neq 0$ or  $a_4+2 a_3\neq 0$.


Conversely, suppose $R_2=0$, $a_3\neq a_6$, and either $a_2\neq0$ or $a_4+2a_3\neq0$. 
We have to prove that the real root of $p_1$ coincides with that of $p_3$.

Assume on the contrary that these real roots do not coincide. In that case, the complex conjugate roots 
of $p_1$ and $p_3$ must coincide. Then there exist some $a,a',b,d\in\mathbb{R}$ such that
\begin{equation}\label{eq:d1d2negativos}\begin{split}
p_1(x)&=-a_6 (x-a) \big((x-b)^2+d^2\big),\\
p_3(x)&=(-a_1a_6+a_2) (x-a') \big((x-b)^2+d^2\big).\\
\end{split}\end{equation}
Equating the coefficients, eliminating the variables $a,a',b,d$, 
and computing the minimal associated prime ideals (see Appendix~\ref{ap:A}), 
we obtain the ideals in \eqref{eq:singular_points} 
(which do not satisfy that $a_3\neq a_6$, and either $a_2\neq0$ or $a_4+2a_3\neq0$), 
and an ideal such that one of its generators is $R_1$. 
By Lemma~\ref{lema:R1}, we conclude. 
\end{proof}

The last case is $p_1,p_3$ of degree three with a unique odd-multiplicity real root, 
and possible double roots. 

\begin{prop}\label{prop:g33r1s}
Assume $p_1,p_3$ have degree three (i.e., $a_6(a_2-a_1a_6)\neq0$)  and one of them has a root of multiplicity two or more. Then 
$p_1$ and $p_3$ have the same odd-multiplicity real root if and only if $R_2=0$ and one of the following statements holds:
\begin{equation}\label{eq:8}
D_1=0,\ D_3<0,\ D_1'\neq 0,\ R_{113}\neq 0,
\end{equation}
\begin{equation}\label{eq:9}
D_1=D_1'=0,\ D_3<0,
\end{equation}
\begin{equation}\label{eq:10}
D_3=0,\ D_1<0,\ D_3'\neq 0,\ R_{133}\neq 0,
\end{equation}
\begin{equation}\label{eq:11}
D_3=D_3'=0,\ D_1<0,
\end{equation}
\begin{equation}\label{eq:12}
D_1=D_1'=D_3=0,\ R_{133}\neq 0,
\end{equation}
\begin{equation}\label{eq:13}
D_1=D_3=D_3'=0,\ R_{113}\neq 0,
\end{equation}
\begin{equation}\label{eq:15}
\begin{split}
 D_1=D_3=0,\ D_1'\neq 0,\ D_3'\neq 0,\ 
\bar R_{113}\neq 0,\ \bar R_{133}\neq 0,
\end{split}
\end{equation}

\end{prop}

\begin{proof}
Since $p_1,p_3$ have degree three, then $a_6\neq 0$, $a_2-a_1a_6\neq 0$.
By Lemma \ref{lema:R1}, $p_1,p_3$ have a real root in common if and only if $R_2=0$.
In the following, we shall assume this to be the case. 

Assume that $p_1$ has a root $x_1$ of multiplicity two or more, and that $p_3$ has a simple 
real root, $x_3$, and two complex conjugate roots, i.e.,  $D_1 = 0$, $D_3<0$.  
The multiplicity of $x_1$ is two if and only if $D_1'\neq 0$, and is three if and only if $D_1'=0$. 
In the former case of $D_1'\neq 0$, $p_1$ 
has a simple root $\bar{x}_1\neq x_1$. Therefore $p_1,p_3$ have the same odd-multiplicity real roots 
if and only if $x_3=\bar{x}_1$. As $R_2=0$, then either $x_3=\bar{x}_1$ or $x_1=\bar{x}_1$. Moreover,
$x_1$ is a root of $p_1'$, while $\bar{x}_1$ is not, so that $x_3=\bar{x}_1$ if and only if $R_{113}\neq 0$.
In the latter case of $D_1'=0$, $x_1$ is the unique real root of $p_1$ with multiplicity three,
and, as $R_2=0$, $x_1=x_3$, so that the odd-multiplicity real roots of $p_1,p_3$ coincide. 

\medskip

Assume that $p_3$ has a root of multiplicity two or more, and $p_1$ has a simple 
real root and two complex conjugate roots. Arguing analogously, we obtain that
the odd-multiplicity real roots of $p_1,p_3$ coincide if and only if
\eqref{eq:10} or \eqref{eq:11} hold.

\medskip

Assume that $p_1$ and $p_3$ have a root of multiplicity two or more, i.e., $D_1=D_3=0$. Firstly, by Proposition~\ref{prop:g33r3}, if both
$p_1$ and $p_3$ have a root of multiplicity three, then it can not be 
common.

If $D_1'=0$, then $p_1$ has a triple root. As $R_2=0$,
this root coincides with one of the roots of $p_3$. 
If $R_{133}\neq 0$, then it coincides with a simple root of $p_3$, and in any other
case ($R_{133}=0$ and $D_3'\neq 0$), it coincides with the double root of $p_3$. 

Analogously, if $D_3'=0$, then the triple root of $p_3$ coincides with the odd-multiplicity 
real root of $p_1$ if and only if $R_{113}\neq 0$.

If  $D_1',D_3'\neq 0$, then $p_1$ and $p_3$ have a root of multiplicity two and a simple root.
In this case, the greatest common divisor of $p_1$ and $p'_1$ is $r_1$, a degree-one polynomial, so that $\bar R_{113}$ is zero if and only if the double root of $p_1$
is a root of $p_3$.  Analogously, $\bar R_{133}$ is zero if and only if 
the double root of $p_3$ is a root of $p_1$.
By Proposition~\ref{prop:g33r3}, if $p_1,p_3$ have a double root
in common, then their simple root is distinct. So $p_1,p_3$ have the same
simple root if and only if $\bar R_{113}\neq 0$ and $\bar R_{133}\neq 0$.
\end{proof}		

Finally, we compute examples of points for some of the semi-varieties and 
their dimensions.

\begin{prop}\label{prop:dims}
The codimensions of the semi-varieties defined by conditions of Theorem~\ref{teo:A} are
the following:
\begin{itemize}
 \item $5a)$ has codimension one.
 \item $5b),5d)$ have codimension two.
 \item $2),3a),4),5f)$ have codimension three. 
 \item $1a),1b)$ have codimension four. 
 \item $3b)$ has codimension five.
 \item $5f)$ has codimension two or three.
 \item $5c),5e),5g),5h)$ have codimension of at least two.
\end{itemize}

\end{prop}
\begin{proof}
In Table~\ref{table:codimensions} we give one point in each of the semi-varieties, such that if the definition 
of the semi-variety contains inequalities then the inequalities hold strictly.

In the same table, we include the codimension of the tangent 
space of the semi-variety at that point, $c_p$. To obtain it, we compute
the rank of the Jacobian matrix of the equations (equalities) defining the semi-variety at that point.
If the rank is maximum (the point is regular), then it coincides with the codimension of 
the variety at that point. (We set it to * if the point is not singular.)

Finally, $c_I$ denotes the (Krull) codimension of the defining ideal $I$ of the smallest variety cointaining the corresponding semi-variety  (i.e., considering the ideal generated only by the polynomials of the equalities). In symbols, $c_I = \mathrm{codim}(\mathcal{V}_\mathbb{C}(I)) := n - \mathrm{dim}(I)$, where $n$ is the number of indeterminates in the base ring (see Appendix A). By Theorem 3.9  $c_p \leq c_P$, where $P$ is a minimal prime of $I$ vanishing at $p$ and the equality holds if the point is regular. Therefore, since the dimension of $I$ is the maximum of the dimensions of its associated prime ideals, if $c$ denotes the (real) codimension of the variety, then $c_p \geq c\geq c_I$ at the regular points.

\begin{table}[ht]
\begin{tabular}{|c|c|c|c|}
\hline
Case & Point & $c_p$ & $c_I$ \\\hline
1a) & $a_1=1$, $a_2=0$, $a_3=1$, $a_4=-3$, $a_5=0$, $a_6=0$. & 4 & 4\\\hline
1b) & $a_1=1$, $a_2=1$, $a_3=-1$, $a_4=2$, $a_5=-4$, $a_6=1$. & 4 & 4\\\hline
2)  & \makecell{$a_1=-1$, $a_2=\sqrt{14}$, $a_3=-2$, \\
$a_4=-1$, $a_5=-3 \sqrt{14}$, $a_6=0$.} & 3 & 3 \\\hline
3a) & \makecell{$a_1=-1$, $a_2=(201 + 2 \sqrt{1509})/58$, \\
$a_3=(-33 + 4 \sqrt{1509})/58$, $a_4=-1$, \\
$a_5=-16$, $a_6=(-201 - 2 \sqrt{1509})/58)$} & 3 & 3 \\\hline
3b) & $a_1=0$, $a_2=0$, $a_3=0$, $a_4=1$, $a_5=-2$, $a_6=-1$. & 5 & 5\\\hline
4)  & $a_1=1$, $a_2=0$, $a_3=1/3$, $a_4=-1$, $a_5=-1$, $a_6=0$. & 3 & 3\\\hline
5a) & \makecell{$a_1=0$, $a_2=1$, $a_3=-15/16$,  \\ 
$a_4=-53/16$, $a_5=(-941 - 31 \sqrt{7913})/512$, $a_6=1$.} & 1 & 1\\\hline
5b) & \makecell{$a_1=0$, $a_2=(4096 - 7 \sqrt{1726})/16384$, $a_3=0$, \\
$a_4=-(58339673 + 28672 \sqrt{1726})/94666752$, \\
$a_5=-1$, $a_6=-2889/16384$.} & 2 & 2\\\hline
5c) & $a_1=1$, $a_2=4$, $a_3=-12$, $a_4=30$, $a_5=-15$, $a_6=1/2$ & * & 2\\\hline
5d) & \makecell{$a_1=0$, $a_2=\sqrt{185}/32$, $a_3=0$,\\ 
$a_4=-1$, $a_5=-3 \sqrt{185}/32$, $a_6=-5/32$} & 2 & 2\\\hline
5e) & $a_1=0$, $a_2=2 \sqrt{2}$, $a_3=-1$, $a_4=0$, $a_5=-9 \sqrt{2}$, $a_6=8$ & * & 2\\\hline
5f) & $a_1=0$, $a_2=2/3$, $a_3=0$, $a_4=-1$, $a_5=-2$, $a_6=-1/3$ & 3 & 2\\\hline
5g) & $a_1=0$, $a_2=1$, $a_3=-9/2$, $a_4=15/2$, $a_5=-15$, $a_6=8$ & * & 2\\\hline
5h) & $a_1=0$, $a_2=1$, $a_3=-8$, $a_4=35/2$, $a_5=-15$, $a_6=9/2$ & * & 2\\\hline
\end{tabular}
\caption{Codimensions of the semi-varieties.}
\label{table:codimensions}
\end{table}
\end{proof}

 
\subsection{Proof of Theorem \ref{teo:B} }
The trigonometric polynomial $B(\theta)$ has definite sign if and only 
if $q(x)\equiv 0$. I.e., the parameters belong to the variety defined
by the ideal obtained by equating the coefficients of $q(x)$ to zero:
\begin{align*}
2 a_1 a_2 + 4 a_3 + a_4=0,\\ 
12 a_2 + 3 a_5 - a_1 (6 a_3 + 2 a_4)=0,\\ 
8 a_3 + 3 a_4 + 4 a_6 + a_1 (6 a_2+2 a_5),\\ 
4 a_2 + a_5 - 2 a_1 a_6=0.
\end{align*}	
Computing the minimal associated prime ideals and a minimal set
of generators (see Appendix~\ref{ap:A}), we obtain three minimal ideals. 
But the first one contains the polynomial $a_1^2+4$, so that the associated variety
is empty. The other two prime ideals obtained are
\begin{equation*}
\langle a_1,4a_2+a_5,a_3-a_6,a_4+4a_6\rangle,
\end{equation*}	
and
\begin{equation*}
\langle a_6,3a_3+a_4,4a_2+a_5,3a_1a_5+2a_4\rangle.
\end{equation*}

\appendix

\section{SINGULAR codes}\label{ap:A}
\begin{verbatim}
// Proposition 4.1;
LIB "primdec.lib";
ring r = 0, (a1,a2,a3,a4,a5,a6,x), dp;
poly p1 = -a6*x^3 - 3*a2*x^2 - a5*x^2 + 3*a3*x + a4*x + a2;
poly p2 = -a2*x^3 + 2*a3*x^2 + a4*x^2 + a6*x^2 + 3*a2*x + a5*x - a3;
poly p3 = a1*p1-p2;
ideal i1 = coeffs(p1,x);
ideal i3 = coeffs(p3,x);
mres(i1,1)[1];
mres(i3,1)[1];
\end{verbatim}

\begin{verbatim}
// Proposition 4.3 Case 2;
// D1<0 implies R113!=0;
LIB "primdec.lib";
ring r = 0, (a1,a2,a3,a4,a5,a6,x), dp;
poly p1 = -a6*x^3-(3*a2+a5)*x^2+(3*a3+a4)*x+a2;
poly p2 = -a2*x^3+(2*a3+a4+a6)*x^2+(3*a2+a5)*x-a3;
poly p3 = a1*p1-p2;
ideal R = resultant(p1,p2,x);
poly R2 = minAssGTZ(R)[1][1];
poly dp1 = diff(p1,x);
ideal j1 = coeffs(p3,x)[4,1],coeffs(p3,x)[3,1], resultant(dp1,p3,x);
ideal j = R2, resultant(dp1,p1,x)+x^2, j1; 
j = sat(j,a6)[1];
list l = minAssGTZ(j); // Takes some time
reduce(x,std(l[1]));
reduce(coeffs(p3,x)[2,1],std(l[2]));
reduce(1+a1^2,std(l[3])); 
\end{verbatim}

\begin{verbatim}
// Proposition 4.3 Case 2;
// p1 with a root of multiplicity three;
LIB "primdec.lib";
ring r = 0, (a1,a2,a3,a4,a5,a6,x,a), dp;
poly p1 = -a6*x^3 - 3*a2*x^2 - a5*x^2 + 3*a3*x + a4*x + a2;
poly p2 = -a2*x^3 + 2*a3*x^2 + a4*x^2 + a6*x^2 + 3*a2*x + a5*x - a3;
poly p3 = a1*p1-p2;
poly R2 = minAssGTZ(resultant(p1,p2,x))[1][1];
poly p1d = p1 + a6*(x-a)^3;
ideal i3 = coeffs(p3,x);
poly p3l = i3[3]*x+i3[4];
ideal i1d = coeffs(p1d,x);
ideal i3d = coeffs(p3l,x);
ideal i13 = i1d,i3d,R2;
ideal ie=eliminate(i13,a);
list J=minAssGTZ(ie);
mres(J[3],1)[1];
\end{verbatim}

\begin{verbatim}
// Lemma 4.7;
LIB "primdec.lib";
ring r = 0, (a1,a2,a3,a4,a5,a6,x), dp;
poly p1 = -a6*x^3 - 3*a2*x^2 - a5*x^2 + 3*a3*x + a4*x + a2;
poly p2 = -a2*x^3 + 2*a3*x^2 + a4*x^2 + a6*x^2 + 3*a2*x + a5*x - a3;
poly R2 = minAssGTZ(resultant(p1,p2,x))[1][1];
ideal sR2 = R2,jacob(R2);
minAssGTZ(sR2); 
\end{verbatim}

\begin{verbatim}
// Proposition 4.8;
LIB "primdec.lib";
ring r = 0, (a1,a2,a3,a4,a5,a6,x,a,ap,b,d), dp;
poly p1 = -a6*x^3 - 3*a2*x^2 - a5*x^2 + 3*a3*x + a4*x + a2;
poly p2 = -a2*x^3 + 2*a3*x^2 + a4*x^2 + a6*x^2 + 3*a2*x + a5*x - a3;
poly p3 = a1*p1-p2;
poly R2 = minAssGTZ(resultant(p1,p2,x))[1][1];
poly p1d = p1 + a6*(x-a)*((x-b)^2-d^2);
poly p3d = p3 + (a1*a6-a2)*(x-ap)*((x-b)^2-d^2);
ideal i1d = coeffs(p1d,x);
ideal i3d = coeffs(p3d,x);
ideal i13 = i1d,i3d,R2;
ideal ie=eliminate(i13,a*ap*b*d);
list J=minAssGTZ(ie); 
\end{verbatim}

\begin{verbatim}
// Proposition 4.10;
LIB "primdec.lib";
ring r = 0, (a1,a2,a3,a4,a5,a6,x), dp;
poly p1 = -a6*x^3 - 3*a2*x^2 - a5*x^2 + 3*a3*x + a4*x + a2;
poly p2 = -a2*x^3 + 2*a3*x^2 + a4*x^2 + a6*x^2 + 3*a2*x + a5*x - a3;
poly p3 = a1*p1-p2;
poly dp1 = diff(p1,x);
poly dp3 = diff(p3,x);
poly ddp1 = diff(dp1,x);
poly ddp3 = diff(dp3,x);
poly D1 = resultant(p1,dp1,x);
poly D1p = resultant(p1,ddp1,x);
poly D3 = resultant(p3,dp3,x);
poly D3p = resultant(p3,ddp3,x);
poly R2 = minAssGTZ(resultant(p1,p2,x))[1][1];
ideal i1a = a6,a5,3*a3+a4,a2;
ideal i1b = a1*a6-a2,a1*a5-a3+a4+a6,a1*(3*a3+a4)-3*a2-a5,a1*a2+a3;
ideal i2  = a6,3*a2+a5,R2;
ideal i3a = a2-a1*a6,2*a3+a4+a1*(3*a2+a5)+a6,R2;
ideal i3b = 4*a4-9*a6,4*a3+5*a6,9*a2+a5,9*a1*a6+a5,8*a1^2-1;
ideal i4  = a2,a6,3*a3+a4;
ideal i5a = R2;
ideal i5b = R2,D1;
ideal i5c = R2,D1,D1p;
ideal i5d = R2,D3;
ideal i5e = R2,D3,D3p;
ideal i5f = R2,D1,D3;
ideal i5g = R2,D1,D1p,D3;
ideal i5h = R2,D1,D3p,D1;
nvars(basering) - dim(std(i1a));
nvars(basering) - dim(std(i1b));
nvars(basering) - dim(std(i2));
nvars(basering) - dim(std(i3a));
nvars(basering) - dim(std(i3b));
nvars(basering) - dim(std(i4));
nvars(basering) - dim(std(i5a));
nvars(basering) - dim(std(i5b));
nvars(basering) - dim(std(i5c));
nvars(basering) - dim(std(i5d));
nvars(basering) - dim(std(i5e)); // Takes some time;
nvars(basering) - dim(std(i5f));
nvars(basering) - dim(std(i5g));
nvars(basering) - dim(std(i5h)); // Takes some time; 
\end{verbatim}

\begin{verbatim}
// Theorem B
LIB "primdec.lib";
ring r = 0, (a1,a2,a3,a4,a5,a6), dp;
poly c0 = 2*a1*a2 + 4*a3 + a4; 
poly c1 = 12*a2 + 3*a5 - a1*(6*a3 + 2*a4); 
poly c2 = 8*a3 + 3*a4 + 4*a6 + a1*(6*a2+2*a5); 
poly c3 = 4*a2 + a5 - 2*a1*a6;
ideal iB = c0,c1,c2,c3;
list LB = minAssGTZ(iB);
mres(LB[1],1)[1];
mres(LB[2],1)[1];
mres(LB[3],1)[1];
\end{verbatim}

\section*{Acknowledgments}

The first two authors were partially supported by AEI/FEDER UE grant number
MTM 2011-22751 and Junta de Extremadura grant GR15055 (Junta de Extremadura/FEDER funds). The third author was partially supported by the research group FQM-024 (Junta de Extremadura/FEDER funds) and by the project MTM2015-65764-C3-1-P (MINECO/FEDER, UE).
The fourth author was partially supported by Junta de Extremadura grant GR15055 (Junta de Extremadura/FEDER funds).

\end{document}